\documentclass[11pt,leqno]{amsart}
\usepackage{amscd}
\usepackage{amssymb}
\usepackage{amsfonts}
\usepackage{latexsym}
\usepackage{verbatim}
\usepackage{amsthm}
\numberwithin{equation}{section}

\theoremstyle{plain}
\newtheorem{theorem}{Theorem}[section]
\newtheorem{definition}[theorem]{Definition}
\newtheorem{lemma}[theorem]{Lemma}
\newtheorem{prop}[theorem]{Proposition}
\newtheorem{cor}[theorem]{Corollary}
\newtheorem{rem}[theorem]{Remark}

\newcommand{\ov}[1]{\overline{#1}}

\begin{document}
\title{Quasi-K\"ahler manifolds with trivial Chern Holonomy}
\author{Antonio J. Di Scala and Luigi Vezzoni}
\date{\today}
\address{Dipartimento di Matematica
\\ Politecnico di Torino \\ Corso Duca degli Abruzzi 24, 10129 Torino,
Italy} \email{antonio.discala@polito.it}
\address{Dipartimento di Matematica \\ Universit\`a di Torino\\
Via Carlo Alberto 10 \\
10123 Torino\\ Italy} \email{luigi.vezzoni@unito.it}
\thanks{This work was supported by the Project M.I.U.R. ``Riemann Metrics and  Differentiable Manifolds''
and by G.N.S.A.G.A. of I.N.d.A.M.}
\begin{abstract}
In  this paper we study
almost complex manifolds admitting a quasi-K\"ahler Chern-flat metric (Chern-flat means that
the holonomy of the Chern connection is trivial). We prove that in the compact case such manifolds are all nilmanifolds.
Some partial classification results are established and
we prove that a quasi-K\"ahler Chern-flat structure can be tamed by a symplectic form if and only if the ambient space is isomorphic to a  flat torus.
\end{abstract}
\maketitle
%%%%%%%%%%%%%%%%%%%%%%%%%%%%%%%%%%%%%%%%%%%%%%%%%%%% SIMBOLI MATEMATICI %%%%%%%%%%%%%%%%%%%%%%%%%%%%%%%%%%%%%%%%%%%%%%%%%%
\newcommand\C{{\mathbb C}}
\newcommand\R{{\mathbb R}}
\newcommand\Z{{\mathbb Z}}
\newcommand\T{{\mathbb T}}
\newcommand\GL{{\rm GL}}
\newcommand\SL{{\rm SL}}
\newcommand\SO{{\rm SO}}
\newcommand\Sp{{\rm Sp}}
\newcommand\U{{\rm U}}
\newcommand\SU{{\rm SU}}
\newcommand{\Gdue}{{\rm G}_2}
\newcommand\re{\,{\rm Re}\,}
\newcommand\im{\,{\rm Im}\,}
\newcommand\id{\,{\rm id}\,}
\newcommand\tr{\,{\rm tr}\,}
\renewcommand\span{\,{\rm span}\,}
\newcommand\Ann{\,{\rm Ann}\,}
\newcommand\Hol{{\rm Hol}}
\newcommand\Ric{{\rm Ric}}
\newcommand\nc{\widetilde{\nabla}}
\renewcommand\d{{\partial}}
\newcommand\dbar{{\bar{\partial}}}
\newcommand\s{{\sigma}}
\newcommand\sd{{\bigstar_2}}
\newcommand\K{\mathbb{K}}
\renewcommand\P{\mathbb{P}}
\newcommand\D{\mathbb{D}}
\newcommand\al{\alpha}
\newcommand\f{{\varphi}}
\newcommand\g{{\frak{g}}}
\renewcommand\k{{\kappa}}
\renewcommand\l{{\lambda}}
\newcommand\m{{\mu}}
\renewcommand\O{{\Omega}}
\renewcommand\t{{\theta}}
\newcommand\ebar{{\bar{\varepsilon}}}
\newcommand{\dis}{\di^{\circledast}}
\newcommand{\fo}{\mathcal{F}}
\newcommand{\op}{\overline{\parzial}}
\newcommand{\wn}{\widetilde{\nabla}}
\section{Introduction}
A K\"ahler manifold is a complex manifold $(M,J)$ equipped with a compatible Hermitian metric $g$ whose Levi-Civita connection preserves $J$. The
condition $\nabla J=0$ imposes a strong constraint on the curvature of $g$. In particular if $(M,J,g)$ is a compact
flat K\"ahler manifold, then it is isomorphic to a flat torus. K\"ahler manifolds can be seen as a subclass
of \emph{almost Hermitian manifolds}.
An almost Hermitian manifold is a smooth manifold equipped with a not-necessarily integrable ${\rm U}(n)$-structure $(g,J)$. Such a
manifold differs from
the K\"ahler one by the non-vanishing of the covariant derivative $\nabla J$, so in the almost Hermitian case it is
useful to study the geometry related to
connections with torsion preserving $g$ and $J$, in place of the geometry of the Levi-Civita connection.
This kind of connection is usually called \emph{Hermitian} (see e.g. \cite{gau}).

Any almost Hermitian manifold $(M,g,J)$
always admits a unique Hermitian connection $\widetilde{\nabla}$ whose torsion has
everywhere vanishing $(1,1)$-part. This connection was introduced by
Ehresmann and Libermann in \cite{EL} and for a K\"ahler manifold coincides with
the Levi-Civita connection. In the complex case, $\widetilde{\nabla}$ coincides with the connection used by Chern in \cite{Chern}
and is called the \emph{Chern connection} (in many papers it is also called the \emph{canonical Hermitian connection},
the \emph{second canonical Hermitian connection}
or simply the \emph{canonical connection}). Such a connection and its associated curvature tensor $\widetilde{R}$
plays a central role in almost Hermitian geometry. For instance, $\widetilde{R}$ appears in the Sekigawa's proof
of  the
Goldberg conjecture for Einstein manifolds of non-negative scalar curvature (see \cite{SEk}) and
recently Tosatti, Weinkove and Yau proved a Donaldson-type conjecture under a positivity assumption on $\widetilde{R}$ (see \cite{yau}).
The role of $\widetilde{\nabla}$ is also important in nearly K\"ahler geometry, since in this special case it
has parallel and totally skew-symmetric torsion (see e.g. \cite{FI,nagyasian} and the references therein).
\\

We recall that a \emph{quasi-K\"ahler}
structure is an almost Hermitian structure whose K\"ahler form $\omega$ satisfies $\overline{\partial}\omega=0$.
The aim of this paper is to study almost complex manifolds admitting a compatible quasi-K\"ahler metric
for which the holonomy associated to its Chern connection is trivial. Such manifolds will be called \emph{quasi-K\"ahler Chern-flat}.
The point is that for quasi-K\"ahler manifolds, the Chern connection can be described by a very simple formula involving  only
the Levi-Civita connection and the almost
complex structure.

The canonical example of a quasi-K\"ahler Chern-flat manifold is the Iwasawa manifold  $M$ equipped with the almost Hermitian
structure associated to the almost complex structure $J_3$ defined
in \cite[p.151]{elsa}. The structure
$J_3$ is the unique invariant almost complex structure on $M$ such that its associated fundamental form $\omega_3$ with respect to a
canonical metric of $M$ is quasi-K\"ahler and not-symplectic. The pair $(M,J_3)$ is a compact almost complex manifold obtained
as a quotient of an almost complex $2$-step nilpotent Lie algebra by a lattice. We prove that all compact quasi-K\"ahler Chern-flat manifolds have this structure.
\begin{theorem}\label{main}
A compact almost complex manifold $(M,J)$ is quasi-K\"ahler Chern-flat if and only if it is isomorphic to a $2$-step nilpotent
 nilmanifold equipped
with a left-invariant almost complex structure whose almost complex Lie algebra $(\g,J)$ satisfies
\begin{equation}\label{stru}
[\g^{1,0},\g^{0,1}]=0\,,\quad [\g^{1,0},\g^{1,0}]\subseteq \g^{0,1}\,,
\end{equation}
being $\g\otimes\C=\g^{1,0}\oplus\g^{0,1}$ the splitting of $\g\otimes \C$ in terms of the eigenspaces of $J$.
\end{theorem}
The proof of Theorem \ref{main} relies upon a theorem of Palais, which can be found in \cite{palais}. In view of Theorem \ref{main},
the problem of classifying
quasi-K\"ahler Chern-flat manifolds reduce to that of classifying $2$-step nilpotent Lie algebras equipped with an almost complex
structure satisfying \eqref{stru}. It is rather natural to call such Lie algebras \emph{quasi-K\"ahler Chern flat}. In section \eqref{examples}
we show that any $n$-dimensional $2$-step nilpotent Lie algebra has a naturally corresponding $2n$-dimensional quasi-K\"ahler Chern flat Lie
algebra. For instance, the Lie algebra associated to the Iwasawa manifold equipped with the almost complex structure $J_3$ is
the quasi-K\"ahler Chern flat Lie algebra associated to the $3$-dimensional Heisenberg group.
In section \eqref{4.2} we classify the $8$-dimensional quasi-K\"ahler Chern Lie algebras
(the $6$-dimensional case was recently carried out
by the authors in \cite{TV}) and in section \eqref{4.3} we describe all the quasi-K\"ahler Chern flat Lie algebras which have
center of complex dimension equal to $1$. At the end of section \ref{QKCFLA}, we
determine the vector space of the infinitesimal deformations of quasi-K\"ahler Chern flat Lie algebras (see section \ref{deformations}).\\

In the last part of the paper we study the problem of taming an almost complex structure admitting a compatible quasi-K\"ahler Chern-flat metric
with a symplectic form. This problem arises from a conjecture of Donaldson and from the work of Tosatti,
Weinkove and Yau (see \cite{donaldson,yau}). We prove the following
\begin{theorem}\label{tosatti}
Let $(M,J)$ a compact quasi-K\"ahler Chern-flat manifold. Assume that there exists a symplectic structure $\omega$  on $M$ taming $J$.
Then $(M,J)$ is a complex torus.
\end{theorem}
\noindent This theorem was recently proved by the authors in the case of the Iwasawa manifold in \cite{TV}.

\bigskip
\noindent {\sc Acknowledgments:} The authors would like to thank Simon Salamon for useful conversations, suggestions and remarks.

\bigskip
\noindent {\sc Notations.}
When a coframe $\{\alpha_1,\dots,\alpha_n\}$ is given we
will denote the $r$-form $\alpha_{i_{1}}\wedge\dots\wedge\alpha_{i_{r}}$ by
$\alpha_{i_{1}\dots i_{r}}$ and
in the indicial expressions the summation sign over repeated indeces is
omitted.
\section{Preliminaries and first results}
Let $(M,J)$ be an almost complex manifold. Then the complexified tangent bundle $T\otimes \C$ of $M$ splits in $T\otimes \C=T^{1,0}\oplus T^{0,1}$.
Consequently the vector space $\Omega^p$ of the complex $p$-forms on $M$ splits as $\Omega^p=\oplus_{r+s=p}\Omega^{r,s}$ and the de Rham differential
operator ${\rm d}$ decomposes accordingly with this type decomposition of complex $p$-form as ${\rm d}=A+\partial+\ov{\partial}+\ov{A}$.
Let us fix now an almost Hermitian metric $g$ on $(M,J)$. It's well known that there exists a unique connection $\widetilde{\nabla}$ on $M$
satisfying the following properties
$$
\widetilde{\nabla} g=0\,,\quad \widetilde{\nabla}J=0\,,\quad {\rm Tor}(\wn)^{1,1}=0\,,
$$
where ${\rm Tor}(\wn)^{1,1}$ denotes the $(1,1)$-part of the torsion of $\widetilde{\nabla}$.
The connection $\wn$ is usually called the \emph{canonical Hermitian connection} or the \emph{Chern connection} of $g$.
The aim of this paper is to study almost complex manifolds admitting an almost Hermitian metric having the
holonomy of the Chern connection trivial. We will adopt the following conventionally definitions:
\begin{definition}
An almost Hermitian metric having trivial Chern holonomy will be called in this paper a \emph{Chern-flat metric}.

We will refer to a \emph{Chern-flat manifold} as to an almost complex manifold admitting an almost Hermitian metric with trivial Chern holonomy.
\end{definition}
As first result on Chern-flat manifolds we have the following
\begin{prop}\label{campiglobali}
An almost complex manifold $(M,J)$ is Chern-flat if and only
if there exists a global $(1,0)$-frame $\{Z_1,\dots,Z_n\}$ on $M$ such that
$[Z_i,Z_{\ov{j}}]=0$, for every $i,j=1,\dots,n$.
\end{prop}
\begin{proof}
Assume that there exists a $J$-compatible almost Hermitian metric $g$ having the holonomy of the Chern connection trivial.
Then there exists a global unitary frame $\{Z_1,\dots,Z_n\}$ on $M$ such that
$$
\widetilde{\nabla}_iZ_{j}=\widetilde{\nabla}_iZ_{\ov{j}}=0\,,\quad i,j=1,\dots,n\,.
$$
Since ${\rm Tor}^{1,1}(\widetilde{\nabla})=0$, then we get
$$
0=\widetilde{\nabla}_iZ_{\ov{j}}-\widetilde{\nabla}_{\ov{j}}Z_{i}-[Z_i,Z_{\ov{j}}]=[Z_i,Z_{\ov{j}}]\,,\quad i,j=1,\dots,n\,.
$$

On the other hand assume that there exists a global $(1,0)$-frame $\{Z_1,\dots,Z_n\}$ such that
$[Z_i,Z_{\ov{j}}]=0$ and let $g$ the metric
$$
g=\sum_{i=1}^n\zeta_{i}\odot\zeta_{\ov{i}}\,,
$$being $\{\zeta_1,\dots,\zeta_n\}$ the dual frame of $\{Z_1,\dots,Z_n\}$. Then we have
$$
0=\widetilde{\nabla}_iZ_{\ov{j}}-\widetilde{\nabla}_{\ov{j}}Z_{i}-[Z_i,Z_{\ov{j}}]=\widetilde{\nabla}_iZ_{\ov{j}}-\widetilde{\nabla}_{\ov{j}}Z_{i}\,,\quad i,j=1,\dots,n\,,
$$
which implies $\widetilde{\nabla}_iZ_{\ov{j}}=0$ since $\widetilde{\nabla}$ preserves $J$. Furthermore, since the frame $\{Z_1,\dots,Z_n\}$ is unitary with respect to $g$, we get
$$
g(\widetilde{\nabla}_iZ_{j},Z_{\ov{k}})=Z_ig(Z_{j},Z_{\ov{k}})-g(Z_{j},\widetilde{\nabla}_iZ_{\ov{k}})=0\,,
$$
which implies $\widetilde{\nabla}_iZ_{j} = 0$, for $i,j=1,\dots,n$. Hence there exists a global frame of type
$(1,0)$ on $M$ which is preserved by $\widetilde{\nabla}$ and
this is equivalent to require ${\rm Hol}(\widetilde{\nabla})=0$.
\end{proof}
\subsection{Left-invariant Chern-flat structures on Lie groups}
Let $G$ be a Lie group equipped with a left-invariant almost complex structure $J$. Then $J$ induces an almost complex structure,
which we denote with the same letter, on the
Lie algebra $\g$ associated to $G$. By \emph{almost complex structure on a Lie algebra $\g$} we just mean an endomorphism $J$ of $\g$ satisfying
$J^{2}=-{\rm I}$. Such an endomorphism has no any relation with the bracket of $\g$. Anyway $J$ allows us to split the complexified of $\g$ in
$\g\otimes\C=\g^{1,0}\oplus\g^{0,1}$. The space $\g^{1,0}$ can be clearly  identified with the space of left-invariant vector fields of type $(1,0)$
on the Lie group $(G,J)$.
In view of Proposition \ref{campiglobali}, the existence of left-invariant almost Hermitian Chern-flat metrics on $(G,J)$
can be characterized in terms of $(\g,J)$.
\begin{prop}\label{proppreliminare}
Let $(G,J)$ be a Lie group equipped with a left-invariant almost complex structure. If the almost complex Lie algebra $(\g,J)$ associated to
$(G,J)$ satisfies
$$
[\g^{1,0},\g^{0,1}]=0\,,
$$
then every left-invariant almost Hermitian metric $g$ on $(G,J)$ is Chern-flat.
\end{prop}
\begin{proof}
Let $g$ be an arbitrary left-invariant almost Hermitian metric on $(G,J)$ and let $\widetilde{\nabla}$ be the Chern
connection of $g$. Since $g$ is left-invariant, then we can find a
left-invariant unitary frame $\{Z_1,\dots,Z_n\}$
on $(G,J)$. Again using ${\rm Tor}(\widetilde{\nabla})=0$ and $\widetilde{\nabla}g=0$ we get that
$\widetilde{\nabla}_iZ_j=\widetilde{\nabla}_iZ_{\ov{j}}$
for every $i,j=1,\dots,n$. Hence ${\rm Hol}(\widetilde{\nabla})=0$.
\end{proof}
Proposition \ref{proppreliminare} justifies the following
\begin{definition}
An almost complex Lie algebra $(\g,J)$ is called \emph{Chern-flat} if it satisfies $[\g^{1,0},\g^{0,1}]=0$.
\end{definition}

Condition to be \emph{Chern-flat} can be also characterized as follows:
\begin{prop}\label{real}
An almost complex Lie algebra $(\g,J)$ is Chern-flat if and only if $J$ satisfies $[JX,Y]=[X,JY]$ for every $X,Y\in\g$.
\end{prop}

Proposition \ref{real} as the following immediate consequence
\begin{cor}
Let $(\g,J)$ be a Chern-flat Lie algebra, then the center of $\g$ is $J$-invariant.
\end{cor}
\begin{proof}
Let $X$ be an arbitrary vector belonging to the center of $\g$ and let $Y\in\g$. Then we have $[JX,Y]=[X,JY]=0$. Hence $JX$ belongs to the
center of $\g$.
\end{proof}
\section{Quasi-K\"ahler Chern-flat manifolds \\
and proof of Theorem \ref{main}}\label{Iwasawa}
In this section we introduce quasi-K\"ahler Chern-flat manifolds and we prove Theorem \ref{main}.\\

We recall that an almost Hermitian metric $g$ on an almost complex manifold  $(M,J)$ is called \emph{quasi-K\"ahler} if the
K\"ahler $2$-form $\omega_g(\cdot,\cdot):=g(J\cdot,\cdot)$ associated to $(g,J)$ is $\ov{\partial}$-closed.
In this case the Chern connection of $g$ is simply the connection described by the formula
$$
\widetilde{\nabla}=\nabla-\frac12 J\nabla J\,,
$$
where $\nabla$ is the Levi-Civita connection associated to $g$ (see for instance \cite{gau}).

\begin{definition}
In this paper we refer to a \emph{quasi-K\"ahler Chern-flat manifold} as to an almost complex manifold admitting a
compatible almost Hermitian metric which is
quasi-K\"ahler and Chern-flat.
\end{definition}

The canonical example of a quasi-K\"ahler Chern-flat manifold is the Iwasawa manifold equipped with the almost
complex structure $J_3$ defined in \cite[p. 151]{elsa}.
This almost complex manifold is defined as follows:\\
\noindent Let $G$ be the $3$-dimensional \emph{complex Heisenberg group}
$$
G:=\left\{
\left(
\begin{array}{cccccc}
1  &z_1    &z_2  \\
0  &1      &z_3     \\
0  &0      &1\\
\end{array}
\right)\,:z_i\in\C\,,\operatorname{i}=1,2,3
\right\}
$$
and let $M$ be the compact manifold $M=G/\Gamma$, where $\Gamma$ is the co-compact lattice of $G$ formed by the matrices with integral entries.
Then $M$ is a $2$-step nilpotent nilmanifold  usually called the \emph{Iwasawa manifold}.
This manifold admits a global frame
$\{X_1,X_2,X_3,X_4,X_5,X_6\}$ satisfying the following structure equations
$$
\begin{aligned}
&[X_1,X_2]=X_3\,,\quad [X_4,X_5]=-X_3
&[X_2,X_4]=X_6\,,\quad [X_5,X_1]=X_6\,.
\end{aligned}
$$
The structure $J_3$ can be defined as the almost complex structure
$$
\begin{aligned}
&J_3X_1=X_4\,,\quad &&J_3X_2=X_5\,,\quad &&J_3X_3=X_6\,,\\
&J_3X_4=-X_1\,,     &&J_3X_5=-X_2\,,     &&J_3X_6=-X_3\,.
\end{aligned}
$$
Such an almost complex structure induces the $(1,0)$-frame
$$
Z_1=X_1-\operatorname{i}X_4\,,\quad Z_2=X_2-\operatorname{i}X_5\,,\quad Z_3=X_3-\operatorname{i}X_6\,.
$$
Clearly
$$
[Z_1,Z_2]=2\,Z_{\overline{3}}\,,\quad [Z_{\overline{1}},Z_{\overline{2}}]=2\,Z_{3}\,.
$$
Let $\{\zeta_1,\zeta_2,\zeta_3\}$ be the coframe associated to $\{Z_1,Z_2,Z_3\}$. Then it easy to show that the
canonical metric
$g=\zeta_{1}\odot\zeta_{\ov{1}}+\zeta_{2}\odot\zeta_{\ov{2}}+\zeta_{3}\odot\zeta_{\ov{3}}$ is quasi-K\"ahler and Chern-flat with respect to $J_3$.
\begin{rem}\emph{The notation used in this paper is slightly different from the one taken into account in \cite{elsa},
where the Lie algebra associated to the complex Heisenberg group is described using the following structure equations
$$
[e_1,e_3]=-e_5\,,\quad [e_2,e_4]=e_5\,,\quad [e_1,e_4]=-e_6\,,\quad [e_2,e_3]=-e_6\,.
$$
The following relations describe an explicit isomorphism between our frame and the one taken into account in \cite{elsa} $$
\begin{aligned}
&X_1 \leftrightarrow e_1\,, \quad X_4 \leftrightarrow -e_2 \,,\\
&X_2 \leftrightarrow e_4 \,,  \quad  X_5 \leftrightarrow e_3 \,,\\
&X_3 \leftrightarrow e_5\,,   \quad  X_6 \leftrightarrow e_6 \,.
\end{aligned}
$$
With respect to the frame $\{e_i\}$,
the fundamental 2-form $\omega_3$ associated to the pair $(g,J_3)$
is \[ \omega_3 = -e^{12} -e^{34} + e^{56} \, ,\]
where $e^{ij}$ means $e^i \wedge e^j$.
In view of \cite{elsa} the almost complex structure $J_3$ is the unique almost complex structure on the Iwasawa manifold
which is quasi-K\"ahler with respect to the metric $g_3$ and it is invariant under the
 natural action of $\T^3$ on the space of the $g_3$-compatible almost Hermitian structures on $M$.}
\end{rem}

\medskip
The Iwasawa manifold is an example of a compact $2$-step nilpotent nilmanifold equipped with a left-invariant almost complex structure which is
quasi-K\"ahler and Chern-flat. A \emph{$2$-step nilpotent nilmanifold} is by definition an homogeneous space which is obtained by a quotient of a
$2$-step nilpotent Lie group $G$ with a discrete subgroup $\Gamma$. Any left-invariant almost complex structure $J$ on a nilmanifold induces an
almost complex structure $J$ on the Lie algebra $\g$ associated to $G$.\\

Now we prove Theorem \ref{main}:

\begin{proof}[Proof of Theorem $\ref{main}$]
Let $(M,J)$ be a compact almost complex manifold and
assume that there exists a quasi-K\"ahler Chern-flat metric $g$ compatible with $J$ on $M$. Then we can find a global unitary frame on $M$
$\{Z_1,\dots,Z_n\}$ satisfying
$$
\widetilde{\nabla}_iZ_j=\widetilde{\nabla}_{\ov{i}}Z_j=0
$$
for every $i,j=1,\dots,n$, where $\widetilde{\nabla}$ denotes the Chern connection of $g$.
Such a condition reads in terms of brackets as
$$
[\widetilde{Z}_i,\widetilde{Z}_{\ov{j}}]=0\,,\quad [\widetilde{Z}_i,\widetilde{Z}_j]\in \Gamma(T^{0,1}M)\,.
$$
We can write $[Z_i,Z_j]=\sum_{r=1}^nc_{ij}^{\ov{r}}Z_{\ov{r}}$.
In order to show that $M$ has a natural
structure of Lie group, it is enough to check that the maps $c_{ij}^{\ov{r}}$'s are constant (see \cite{palais}). The Jacobi identity imply
$$
\begin{aligned}
0=&\mathfrak{S}[[Z_i,Z_j],Z_{\ov{k}}]=[[Z_i,Z_j],Z_{\ov{k}}]=[c_{ij}^{\ov{r}}Z_{\ov{r}},Z_{\ov{k}}]\\
=&c_{ij}^{\ov{r}}[Z_{\ov{r}},Z_{\ov{k}}]-
Z_{\ov{k}}(c_{ij}^{\ov{r}})Z_{\ov{r}}
=c_{ij}^{\ov{r}}c_{\ov{r}\ov{k}}^{l}Z_{l}-Z_{\ov{k}}(c_{ij}^{\ov{r}})Z_{\ov{r}}\,,\quad i,j,k=1,\dots,n\,,
\end{aligned}
$$
where the symbols $\mathfrak{S}$ denotes the cyclic sum.
Then \[ c_{ij}^{\ov{r}}c_{\ov{r}\ov{k}}^{l} = 0 \,,\quad  Z_{\ov{k}}(c_{ij}^{\ov{r}}) = 0 \, .\]

Thus, the functions $c_{ij}^{\ov{r}}$'s are holomorphic maps on $M$. Since $M$ is compact, the $c_{ij}^{\ov{r}}$'s are constant and, consequently,
$M$ inherits a natural structure of homogeneous space obtained as a quotient of a Lie group with a lattice.
Since the finite dimensional Lie algebra generated by the fields $\{Z_1,\dots,Z_n\}$ is $2$-step nilpotent we get that $M$ is indeed a $2$-step nilpotent nilmanifold.
The almost structure $J$ is clearly left-invariant on $M$.  Note that the almost complex Lie algebra associated to the universal cover of $M$ satisfies conditions \eqref{stru}.

On the other hand assume that $M=G/\Gamma$ is a $2$-step nilpotent nilmanifold equipped with a left-invariant almost complex structure $J$
such that the almost complex Lie algebra $(\g,J)$ associated to $(G,J)$ satisfies $\eqref{stru}$. Fix an arbitrary $(1,0)$-coframe
$\{\zeta_1,\dots,\zeta_n\}$ on $(\g,J)$ and let
$$
g:=\sum_{i=1}^n\zeta_i \odot \zeta_{\ov{i}}\,.
$$
Then $g$ induces an almost Hermitian metric on $(M,J)$. Direct computations give that the metric $g$ is quasi-K\"ahler and Chern-flat.
\end{proof}

\begin{rem}\emph{The condition ${\rm Hol}(\widetilde{\nabla})=0$ is quite strong. For instance we have the following results:}
\begin{itemize}
\item \emph{in almost K\"ahler manifolds and in nearly K\"ahler manifolds condition ${\rm Hol}(\widetilde{\nabla})=0$
implies the integrability (see \cite{TV, LV3})};
\vspace{0.1cm}
\item \emph{in dimension $4$ there are no strictly quasi-K\"ahler Chern-flat metrics (see \cite{TV}).}
\end{itemize}
\end{rem}

\section{Quasi-K\"ahler Chern-flat Lie algebras}\label{QKCFLA}
In view of Theorem \ref{main},
the problem of classify quasi-K\"ahler Chern-flat manifolds turns to the one of
classify almost complex Lie algebras satisfying equations \eqref{stru}. It is quite natural to introduce the following definition:
\begin{definition} An almost complex Lie algebra $(\g,J)$ satisfying equations \eqref{stru}
is said to be \emph{quasi-K\"ahler and Chern-flat}.
\end{definition}
The following easy-prove Proposition will be useful
\begin{prop} \label{equivaChernAlgebra}
Let $(\g,J)$ be an almost complex Lie algebra. The following facts are equivalent:
\begin{enumerate}
\item[1.] $(\g,J)$ is quasi-K\"ahler and  Chern-flat;
\vspace{0.1cm}
\item[2.] $\partial\Lambda^{1,0}\g^*=\ov{\partial}\Lambda^{1,0}\g^*=0$;
\vspace{0.1cm}
\item[3.] the almost complex structure $J$ satisfies \[ J[X,Y]=-[JX,Y]=-[X,JY] \]
for every $X,Y\in\g$.
\end{enumerate}
\end{prop}
The following result can be viewed as a corollary of Theorem \ref{main}.
\begin{cor}
Let $(M,J)$ be a quasi-K\"ahler Chern-flat manifold. Then every left-invariant almost Hermitian
metric on $(M,J)$ is quasi-K\"ahler and Chern-flat.
\end{cor}
\begin{proof}
Let $g$ be a left-invariant almost Hermitian metric  on $(M,J)$, then we can find a left-invariant $(1,0)$-coframe
$\{\zeta_1,\dots,\zeta_n\}$ with respect to which  $g$ writes as $g=\sum_{i=1}^n\zeta_i\odot\zeta_{\ov{i}}$. Now using \eqref{stru}
it easy to see that the
frame $\{Z_1,\dots,Z_n\}$  dual to $\{\zeta_1,\dots,\zeta_n\}$ satisfies
$$
\widetilde{\nabla}_{i}Z_j=\widetilde{\nabla}_{\ov{i}}Z_j=0\,,\quad i=1,\dots,n\,.
$$
This implies that every left-invariant almost Hermitian metric $g$ on $(M,J)$ is quasi-K\"ahler and Chern-flat.
\end{proof}

\subsection{Examples of quasi-K\"ahler Chern-flat Lie algebras}\label{examples}
In this subsection we show that
it is possible to construct a quasi-K\"ahler Chern-flat Lie algebra
$\mathfrak{h}_{\overline{\mathbb{C}}}$ starting from
from an arbitrary $2$-step nilpotent real Lie algebra $\mathfrak{h}$.\\

Let $(\mathfrak{h},[\,,\,]_{\mathfrak{h}})$ be a $2$-step nilpotent Lie algebra and let $(\mathfrak{h}_{\overline{\mathbb{C}}}=\mathfrak{h}\otimes \C,J_{\mathfrak{h}})$ be the almost complex Lie algebra having the bracket
defined by the relation
$$
[X\otimes a,Y\otimes b]=\ov{ab}\,[X,Y]_\mathfrak{h}
$$
and as almost complex structure the one defined by
$$
J_{\mathfrak{h}}(X\otimes a)=X\otimes \operatorname{i}a\,.
$$
\begin{prop} $(\mathfrak{h}_{\overline{\mathbb{C}}}=\mathfrak{h}\otimes \C,J_{\mathfrak{h}})$ is a quasi-K\"ahler Chern-flat Lie algebra.
\end{prop}
\begin{proof}  Since $\mathfrak{h}$ is a $2$-step nilpotent Lie algebra, then we have
$[[X \otimes a , Y \otimes b], Z \otimes c] = ab \overline{c}\,[[X,Y],Z] = 0$. So $(\mathfrak{h}_{\overline{\mathbb{C}}}=\mathfrak{h}\otimes \C,J_{\mathfrak{h}})$ is also a $2$-step nilpotent
Lie algebra. To finish the proof it is enough to observe that the condition $3$ of Proposition \ref{equivaChernAlgebra} hold by the construction of the bracket.
\end{proof}
Now we give an example of the above construction. Let $\mathfrak{h}= {\rm span}_{\R} \{ X,Y,Z \}$ be the $3$-dimensional real
Heisenberg Lie algebra, i.e. $\mathfrak{h}$ is the $2$-step nilpotent Lie algebra with bracket $[X,Y] = Z$. Let
$(\mathfrak{h}_{\overline{\C}},J_{\mathfrak{h}})$ be its associated quasi-K\"ahler Chern-flat Lie algebra. Then $\mathfrak{h}_{\overline{\C}}$ is the Lie algebra of the Iwasawa manifold described in Section \ref{Iwasawa} equipped with
the almost complex structure $J_3$. An explicit isomorphism between the two Lie algebras can be described as follows:
\[ X_1 \equiv X \otimes 1   \, \, , \, \, X_4 \equiv X \otimes \operatorname{i} , \]
\[ X_2 \equiv Y \otimes 1   \, \, , \, \, X_5 \equiv Y \otimes  \operatorname{i} , \]
\[ X_3 \equiv Z \otimes 1   \, \, , \, \, X_6 \equiv Z \otimes  \operatorname{i}\, . \]
\subsection{Quasi-K\"ahler Chern-flat Lie algebras in low dimension}\label{4.2}
In this subsection we take into account quasi-K\"ahler
Chern-flat Lie algebras in low dimension. In complex dimension $2$ there exists only the abelian case. In view of
\cite{TV} in complex dimension $3$ there are the abelian Lie algebra and the Lie algebra of the complex Heisenberg group. The $4$-dimensional case
is described by the following
\begin{theorem}
Let $(\g,J)$ be a non-complex quasi-K\"ahler Chern-flat Lie algebra of complex dimension $4$. Then there exists a $(1,0)$-frame $\{Z_1,\dots,Z_4\}$ on $\g$  such that
$$
[Z_1,Z_2]=Z_{\ov{3}}
$$
and the other brackets involving the vectors of the frame vanish.
\end{theorem}
\begin{proof}
Let $\{Z_1,\dots,Z_4\}$ be an arbitrary complex frame of type $(1,0)$ on $\g$.
Since $J$ is not integrable, there exists at least a bracket involving two vectors of the frame different from zero.
We may assume that
$[Z_1,Z_2]\neq 0$ and
we can write $[Z_1,Z_2]=\sum_{i=1}^4 c_{12}^{\ov{i}}Z_{\ov{i}}$. The Jacobi identity implies that it is not possible that
$[Z_1,Z_2]\in{\rm Span}_{\C}\{Z_1,Z_2\}$. Indeed if $[Z_1,Z_2]=c_{12}^{\ov{1}}Z_{\ov{1}}+c_{12}^{\ov{2}}Z_{\ov{2}}$, then
$$
[[Z_1,Z_2],Z_{\ov{1}}]=0\,,\quad [[Z_1,Z_2],Z_{\ov{2}}]=0
$$
imply $[Z_1,Z_2]=0$ which is a contradiction. So we can assume that $c_{12}^{\ov{3}}\neq 0$. Replacing $Z_{3}$ with
$Z_{3}=\frac{1}{c_{\ov{12}}^{3}}\,(Z_{3}-\sum_{i\neq 3}c_{\ov{1}\ov{2}}^{i}Z_{i})$, we obtain
$$
[Z_1,Z_2]=Z_{\ov{3}}\,.
$$
Now we observe that the Jacobi identity implies that the vector
$Z_{3}$ belongs to the center of $\g$ since
$$
0=[[Z_1,Z_2],Z_{\ov{k}}]=[Z_{\ov{3}},Z_{\ov{k}}]\,,\quad k=1,\dots,4\,.
$$
The last step consists to show that we may assume $[Z_1,Z_4]=[Z_2,Z_4]=0$. Setting $[Z_1,Z_4]=c_{14}^{\ov{s}}Z_{\ov{s}}$, we have
$$
\begin{aligned}
0&=[[Z_1,Z_4],Z_{\ov{1}}]=-c_{14}^{\ov{2}}Z_{\ov{3}}-c_{14}^{\ov{4}}[Z_{\ov{1}},Z_{\ov{4}}]\,,\\
0&=[[Z_2,Z_4],Z_{\ov{2}}]=c_{14}^{\ov{1}}Z_{\ov{3}}-c_{24}^{\ov{4}}[Z_{\ov{2}},Z_{\ov{4}}]\,,
\end{aligned}
$$
i.e.
$$
\begin{aligned}
&c_{14}^{\ov{2}}Z_{\ov{3}}=-c_{14}^{\ov{4}}[Z_{\ov{1}},Z_{\ov{4}}]\,,\quad
&c_{14}^{\ov{1}}Z_{\ov{3}}=c_{24}^{\ov{4}}[Z_{\ov{2}},Z_{\ov{4}}]
\end{aligned}
$$
which imply $[Z_{1},Z_{4}]=c_{14}^{\ov{1}}Z_{\ov{1}}+c_{14}^{\ov{3}}Z_{\ov{3}}$,
$[Z_2,Z_4]=c_{24}^{\ov{2}}Z_{\ov{2}}+c_{24}^{\ov{3}}Z_{\ov{3}}$.
Hence if $[Z_1,Z_4]\neq 0$, then it is
necessary to be $c_{14}^{\ov{3}}\neq 0$. So in the case $[Z_1,Z_4]\neq 0$ if we replace $Z_4$ with $\frac{1}{c_{14}^{\ov{3}}}Z_4-Z_3$
we force $[Z_1,Z_4]$ to be a multiple of
$Z_{\ov{1}}$ and the Jacobi identity implies that $[Z_1,Z_4]=0$. Analogously if $[Z_2,Z_4]\neq 0$, then one replaces $Z_4$ with
$\frac{1}{c_{24}^{\ov{3}}}Z_4-Z_3$ obtaining $[Z_1,Z_4]=[Z_2,Z_4]=0$.
\end{proof}
%\begin{cor}
%Let $(M,J)$ be an $8$-dimensional compact strictly almost complex manifold admitting a complete $\widetilde{R}$-flat quasi-K\"ahler metric $g$.
%Then $(M,J)$ is isomorphic to the product of the Iwasawa manifold with the standard complex $2$-torus.
%\end{cor}c
In view of this last theorem, any Chern-flat quasi-K\"ahler Lie algebra of complex dimension $4$ is reducible.
In complex dimension $5$ things work differently:
for instance the almost complex Lie algebra $(\g,J)$ admitting a $(1,0)$-frame $\{Z_1,\dots,Z_5\}$ satisfying the following structure equations
$$
[Z_1,Z_2]=Z_{\ov{3}}\,,\quad [Z_2,Z_4]=Z_{\ov{5}}
$$
gives rise an example of a irreducible quasi-K\"ahler Chern-flat Lie algebra.

\subsection{Chern-flat Lie algebras with small center}\label{4.3}
In this subsection we study quasi-K\"ahler Chern-flat Lie algebras having the center of complex dimension equal to one. Such a condition is admissible
when the Lie algebra has complex dimension odd, only and in this case the space reduces to a standard model:
\begin{theorem}\label{centro1}
Let $(\g,J)$ be a quasi-K\"ahler Chern-flat Lie algebra having center $\mathfrak{z}$ of complex dimension one.
Then the complex dimension of $\g$ is odd. Moreover, there exists a $(1,0)$-frame
$\{Z_1,\dots,Z_n\}$ on $\g$  such that
$$
[Z_i,Z_j]=Z_{\ov{n}}\,,\,\,\mbox{for }i,j=1,\dots,n-1,\,\,i< j\,.
$$
where $Z_{\ov{n}}$ is a generator of $\mathfrak{z}^{(0,1)}$.
\end{theorem}
\begin{proof}
Since $(\g,J)$ is Chern-flat, then the center of $\g$ is $J$-invariant.
Let $\g = \mathfrak{v} \oplus \mathfrak{z}$ be a decomposition $\g$, where $\mathfrak{v}$ is a $J$-invariant subspace of $\g$.
Then $\g^{1,0}$ splits in
$\g^{1,0}= \mathfrak{v}^{1,0}\oplus\mathfrak{z}^{1,0}$.
The Lie bracket $[\,\,,\,]$ gives rise to a non-degenerated skew-symmetric two form $\omega$ on $\mathfrak{v}^{1,0}$. Namely,
\[ [X,Y] := \omega(X,Y) \ov{A} \, \, \]
where $X,Y \in \mathfrak{v}^{1,0}$ and $A$ is a generator of $\mathfrak{z}^{1,0}$. It is clear that $\omega$ must be non-degenerated
since otherwise $\dim_\C\mathfrak{z}> 1$. This imply that $\dim_\C\mathfrak{v}$ is even. Let $\dim_\C\mathfrak{v}=2m$, then $\dim_\C\g =2m + 1$.
Now we consider the following skew-symmetric $2$-form on $\mathbb{C}^{2k}$:

\[ \Omega_k = \sum_{\ 1 \leq i < j \leq 2k} dz_i \wedge dz_j  \, \,\]
Notice that \[ \Omega_k = \Omega_{k-1} + \alpha \wedge \beta + dz_{2k-1} \wedge dz_{2k} \, , \]
where $\alpha = \sum_{\ 1 \leq i \leq 2k-2} dz_i$ and $\beta = dz_{2k-1} + dz_{2k}$. Then
\[ \Omega_k ^ k = \underbrace{\Omega  \wedge \cdots \wedge \Omega}_{k-times} = \Omega_{k-1}^{k-1} \wedge dz_{2k-1} \wedge dz_{2k} \, . \]
It follows that $\Omega_k ^ k = \bigwedge_{i=1}^{i=2k} dz_i$. So $\Omega_k$ is not degenerated for all $k \in \mathbb{N}$. Since any two
non-degenerated skew-symmetric $2$-forms on $\C^{2k}$ are equivalent,
we get that there exists a complex basis $(Z_1, \cdots Z_{2k})$ of $\mathfrak{v}^{1,0}$ such that:
\[ \omega = \sum_{i < j} Z_i \wedge Z_j \, .\]
Hence $\{Z_1,\dots,Z_{n-1},Z_n=A\}$ is the desired $(1,0)$-frame.
\end{proof}

\subsection{Deformations of quasi-K\"ahler Chern-flat structures on Lie algebras}\label{deformations}
Let $\g$ be a $2$-step nilpotent Lie algebra
and let $\mathcal{M}$ be the moduli space of Chern-flat quasi-K\"ahler
almost complex structures on $\g$. The goal of this section is to determine the virtual tangent
space to  $\mathcal{M}$ at an arbitrary $[J]\in\mathcal{M}$. Let
$\{J_t\}_{t\in(-\delta,\delta)}$ a family of left-invariant Chern-flat quasi-K\"ahler almost
complex structures on $\g$ satisfying $J_0=J$ and let $L:=\frac{{\rm d}}{{\rm d} t}J_{t|t=0}$.
Since for any $t$ we have
$$
J_t^2=-{\rm I}\,,\quad J_t[X,Y]=-[J_tX,Y]=-[X,J_tY]\,,
$$
then $L$ satisfies
\begin{equation}\label{sopra}
LJ=-JL\,,\quad L[X,Y]=-[LX,Y]=-[X,LY]
\end{equation}
for every $X,Y\in \g$. Moreover, if $\phi_t$ is a smooth family of endomorphisms of $\g$, then
$$
\frac{{\rm d}}{{\rm d} t}\phi_{t}J\phi^{-1}_{t\quad\!\!|t=0}=\mathcal{L}_X J\,,
$$
being $\mathcal{L}$ the Lie derivative and $X$ the infinitesimal vector associated to the path $\phi_t$.  We have
$$
(\mathcal{L}_X J)Y=[X,JY]-J[X,Y]=-2J[X,Y]\,,
$$
for every $X,Y\in\g$.
Hence the virtual tangent space to the moduli space $\mathcal{M}$ at an arbitrary point $[J]$ is defined by
$$
T_{[J]}\mathcal{M}=\frac{\{L\in{\rm End}(\g)\,\,|\,\,LJ=-JL\,,\, L[X,Y]=-[LX,Y]=-[X,LY]\,\,\forall X,Y\in\g\}}{\{[X,\cdot\,\,]\,\,|\,\,X\in\g\}}\,.
$$

Note that conditions \eqref{sopra} imply that if $\mathfrak{z}$ denotes the center of $\g$, then
$$
L([\g,\g])=0\,,\quad L(\g)\subseteq \mathfrak{z}\,.
$$
Indeed, if $Z\in\g^{1,0}$ then $L(Z)\in\g^{0,1}$ and if $Z_1,Z_2$ are vectors of type $(1,0)$ one has
$$
L[Z_1,Z_2]=[L(Z_1),Z_2]=0
$$
and if $X\in\mathfrak{c}$, then
$$
[L(X),Y]=-L[X,Y]=0
$$
for any  $Y\in\g$.
\begin{rem}\emph{
Assume that ${\dim}_\C \mathfrak{z}=1$ and let $L$ be an endomorphism of $\g$ satisfying \eqref{sopra}. Then it has to be
$$
L(Z)=l(Z)\,\ov{A}\,,\quad \forall Z\in\g^{1,0}
$$
for some $l\in\g^*$, where $A$ is a fixed generator of $\mathfrak{z}^{1,0}$.
Since $\mathfrak{z}$ is a $J$-invariant subspace  of $\g$, we can
split  $\g$ in  $\g=\mathfrak{v}\oplus\mathfrak{z}$, being $\mathfrak{v}$ a $J$-invariant
complement of $\mathfrak{z}$. The form $l$ can be viewed as a $1$-form on $\mathfrak{v}$. Since the bracket of $\g$ can be identified with a
non-degenerate $2$-form $\omega$ on $\mathfrak{v}$ by the relation $[Z_1,Z_2]=\omega(Z_1,Z_2)\,\ov{A}$, then there exists a vector $X_L$ of  $\g$
such that  $l(Y)=[X,Y]$ for every $Y\in \mathfrak{g}$. This shows  that in this case the vector space $T_{[J]}\mathcal{M}$ is trivial accordingly to
Theorem \ref{centro1}.
}
\end{rem}

\section{Proof of Theorem \ref{tosatti}}
In \cite{TV} we proved that the almost complex structure $J_3$ on the Iwasawa manifold can not be tamed by a symplectic form.
Theorem \ref{tosatti} is the analogue of this result in an arbitrary Chern-flat quasi-K\"ahler manifold. Theorem \ref{tosatti} can be viewed as a
consequence of the following
\begin{lemma}\label{lemmatosatti}
Let $(M,J)$ be a compact Chern-flat quasi-K\"ahler manifold and let
$\beta\in\Omega^{2,0}$ be a $(2,0)$-form satisfying
$$
\overline{\partial}\beta+A\ov{\beta}=0\,,
$$
then $\beta$ is closed, i.e.
$$
{\rm d}\beta=0\,.
$$
\end{lemma}
\begin{proof}
In view of Theorem \ref{main}, $J$ induces a structure of nilmanifold on $M$. Let $G$ be the universal cover of $M$ and let $(\g,J)$ be the almost
complex Lie algebra associated to  $G$. We can find a $(1,0)$-frame
$\{Z_1,\dots,Z_l,W_1,\dots,W_m\}$ on $\g$ such that
$\{W_1,\dots,W_m,W_{\ov{1}},\dots,W_{\ov{m}}\}$ is a frame of the complexified of the center of $\g$. With respect to this frame we have
$$
[Z_i,Z_j]=\sum_{k=1}^m c_{ij}^{\ov{k}}\,W_{\ov{k}}\,,\quad i,j=1,\dots, l\,.
$$
Let
$\{\zeta_1,\dots,\zeta_l,\eta_1,\dots,\eta_m\}$ be the coframe dual to $\{Z_1,\dots,Z_l,W_1,\dots,W_m\}$; then
$$
{\rm d}\zeta_j=0\,,\quad j=1,\dots,l\,;\quad {\rm d}\eta_i=c^{i}_{\ov{r}\ov{s}}\,\zeta_{\ov{r}\ov{s}}\,,\quad  i=1,\dots,m\,.
$$
Identifying $\beta$ with its pull-back on the universal cover of $M$, we can write
$$
\beta=a_{ij}\zeta_{ij}+a_{ij}b_{ij}\zeta_i\wedge\eta_j+ d_{ij}\eta_{ij}
$$
and
$$
\ov{\beta}=
\ov{a_{ij}}\wedge\zeta_{\ov{i}\ov{j}}+\ov{b_{ij}}\zeta_{\ov{i}}\wedge\eta_{\ov{j}}+\ov{d_{ij}}\eta_{\ov{i}\ov{j}}\,,
$$
where $a_{ij},b_{ij},d_{ij}$ are smooth maps on $M$.
Since the coframe $\{\zeta_1,\dots,\zeta_l,\eta_1,\dots,\eta_m\}$ is $(\partial+\ov{\partial})$-closed, we have
$$
\ov{\partial}\beta=\ov{\partial}(a_{ij})\wedge\zeta_{ij}+\ov{\partial}(b_{ij})\wedge\zeta_i\wedge\eta_j+\ov{\partial}(d_{ij})\wedge\eta_{ij}\,.
$$
Furthermore
$$
A\ov{\beta}=\ov{b_{rs}}c^{\ov{s}}_{ij}\,\zeta_{\ov{r}ij}+2d_{\ov{rs}}c^{\ov{s}}_{ij}\,\eta_{\ov{r}}\wedge\zeta_{ij}=
(\ov{b_{rs}}c^{\ov{s}}_{ij}\,\zeta_{\ov{r}}+2d_{\ov{rs}}c^{\ov{s}}_{ij}\,\eta_{\ov{r}})\wedge\zeta_{ij}\,.
$$
Hence equation $\overline{\partial}\beta+A\ov{\beta}=0$ reads in terms of $a,b,d$'s as
$$
\begin{cases}
& \ov{\partial}(a_{ij})=\ov{b_{rs}}c^{\ov{s}}_{ij}\,\zeta_{\ov{r}}+2d_{\ov{rs}}c^{\ov{s}}_{ij}\,\eta_{\ov{r}}\\
& \ov{\partial}(b_{ij})=\ov{\partial}(d_{ij})=0
\end{cases}
$$
Since $M$ is compact the $b_{ij}$'s and the $d_{ij}$'s have to be constant on $M$. Moreover the $a_{ij}$'s satisfies
$$
\partial\ov{\partial}(a_{ij})=0\,.
$$
Again by the compactness of $M$ we get the $a_{ij}$'s are constant. Hence $\beta$ is a left-invariant form. Consequently it satisfies
$$
\partial\beta=\ov{\partial}\beta=0\,.
$$
Finally equation $\overline{\partial}\beta+A\ov{\beta}=0$ implies $A\beta=0$. Hence $\beta$ is a closed form on $M$.
\end{proof}
\noindent Now we can prove Theorem \ref{tosatti}
\begin{proof}[Proof of Theorem $\ref{tosatti}$]
Since by hypothesis the metric induced on $M$ by the pair $(g,J)$ is quasi-K\"ahler and Chern-flat, then $J$ induces a structure
of nilmanifold on $M$.
Let $(\g,J)$ be the Lie algebra associated to  $(M,J)$ and let $\omega_g$ be the K\"ahler form of $(g,J)$.
Then it has to be
$$
\omega=\omega_g+\beta+\ov{\beta}\,.
$$
being $\beta$ a $(2,0)$-form on $M$. Since $\omega_g$ is $\ov{\partial}$-closed, condition on $\omega$ to be closed in terms
of $\omega_g$ and $\beta$ reads as
$$
\begin{cases}
A\omega_g+\partial\beta=0\\
\overline{\partial}\beta+A\ov{\beta}=0\,.
\end{cases}
$$
In view of Lemma \ref{lemmatosatti} such equations imply ${\rm d}\omega_g=0$. There follows that $g$ is an almost-K\"ahler metric on $(M,J)$. Since $g$ is
Chern-flat, then $J$ is integrable (see \cite{TV}). Hence $g$ is a K\"ahler metric and $M$ is a torus.
\end{proof}

\end{document}